\documentclass[12pt]{article}
\usepackage{graphicx}
\usepackage{geometry}
\usepackage{authblk}
\usepackage{tikz}
\usepackage{subcaption}
\usepackage{hyperref}
\hypersetup{
    colorlinks=true, 
    linkcolor=blue, 
    urlcolor=black, 
    citecolor=blue
}

\usepackage{amsmath}
\usepackage{amsthm}
\usepackage{amssymb}
\usepackage{algorithm}
\usepackage{algpseudocode}
\usepackage{natbib}

\title{Laplacian Spectrum and Domination in Trees}
\author{Deepak Rajendraprasad\footnote{\href{mailto:deepak@iitpkd.ac.in}{deepak@iitpkd.ac.in}} and Durga R. Sankaranarayanan\footnote{\href{mailto:112514003@smail.iitpkd.ac.in}{112514003@smail.iitpkd.ac.in}}\footnote{
A part of this work was done while the second author was a student at Indian Institute of Science Education and Research, Tirupati.}}
\affil{Indian Institute of Technology Palakkad, Palakkad, India}
\date{}

\newtheorem{theorem}{Theorem}
\newtheorem{lemma}[theorem]{Lemma}
\newtheorem{observation}[theorem]{Observation}
\newtheorem{corollary}[theorem]{Corollary}

\theoremstyle{definition}
\newtheorem{definition}{Definition}

\theoremstyle{remark}
\newtheorem*{remark}{Remark}
\newtheorem{case}{Case}

\begin{document}

\maketitle

\begin{abstract}

For a finite simple undirected graph $G$, let $\gamma(G)$ denote the size of a smallest dominating set of $G$ and $\mu(G)$ denote the number of eigenvalues of the Laplacian matrix of $G$ in the interval $[0,1)$, counting multiplicities. Hedetniemi, Jacobs and Trevisan [Eur.\ J.\ Comb. 2016] showed that for any graph $G$, $\mu(G) \leqslant \gamma(G)$. Cardoso, Jacobs and Trevisan [Graphs Combin.\ 2017] asks whether the ratio $\gamma(T)/\mu(T)$ is bounded by a constant for all trees $T$. We answer this question by showing that this ratio is less than $4/3$ for every tree. We establish the optimality of this bound by constructing an infinite family of trees where this ratio approaches $4/3$. We also improve this upper bound for trees in which all the vertices other than leaves and their parents have degree at least $k$, for every $k \geqslant 3$. We show that, for such trees $T$, $\gamma(T)/\mu(T) < 1 + 1/((k-2)(k+1))$. 

\end{abstract}

\section{Introduction}
Let $G$ be a simple undirected graph on the vertex-set $V(G) = \{v_1, \ldots, v_n\}$. The \emph{adjacency matrix} $A(G) = [a_{ij}]$ of $G$ is an $n \times n$ matrix where the entry $a_{ij}$ is $1$ if there is an edge between vertices $v_i$ and $v_j$, and $0$ otherwise. The \emph{Laplacian matrix} of $G$ is $L(G) = D(G) - A(G)$, where $D(G)$ is a diagonal matrix whose $i$-th diagonal entry is the degree of $v_i$. The multi-set of eigenvalues of $L(G)$ is termed the \emph{Laplacian spectrum} of $G$. The Laplacian spectrum of $G$ is always a subset of $[0,n]$. Let $m_GI$ denote the number of Laplacian eigenvalues of $G$ within a real interval $I$, counting multiplicities. The value of $m_G I$ for different choices of $I$ bound many combinatorial parameters of the graph $G$. In this paper we study how $m_G[0,1)$ bounds  the domination number of trees. Henceforth, we abbreviate $m_G[0,1)$ as $\mu(G)$. 

A set of vertices $D$ in $G$ is called \emph{dominating} if every vertex of $G$ is either in $D$ or is adjacent to a vertex in $D$. The \emph{domination number} $\gamma(G)$ is the size of a smallest dominating set in $G$. The connection between Laplacian spectrum and the domination number of a graph was first investigated by Hedetniemi, Jacobs, and Trevisan \cite{hedetniemi2016domination} in 2016 who established the following lemma.

\begin{lemma}[\cite{hedetniemi2016domination}] \label{lem:ubmu}
    For every graph $G$, $\mu(G)\leqslant \gamma(G)$.
\end{lemma}

This study was continued by Cardoso, Jacobs, and Trevisan \cite{cardoso2017laplacian}. Along with $\mu(G) = m_G[0,1)$, they also studied the dependence of $\gamma(G)$ on $\nu(G) = m_G[2,n]$. They proved that $\nu(G) \geqslant \gamma(G)$ for graphs without isolated vertices. They further demonstrated that ratios $\gamma(G)/\mu(G)$ and $\nu(G)/\gamma(G)$ can become arbitrarily large. In contrast, they showed that for trees and graphs which can be made acyclic by removing a bounded number of edges, the ratio $\nu(G)/\gamma(G)$ is bounded. The study concluded with the problem of determining whether the ratio $\gamma(T)/\mu(T)$ is bounded for trees. The main contribution of this paper is the following affirmative answer to this question. 

\begin{theorem}\label{thm:1}
   For every tree $T$,
   \[1\leqslant \frac{\gamma(T)}{\mu(T)}<\frac{4}{3}.\]
\end{theorem}

We show the optimality of this upper bound by constructing an infinite family of trees for which the ratio approaches $4/3$ (c.\ f. Section \ref{sec:example}.) So far, we did not know of any family of trees where this ratio was larger than $25/24$ \cite{cardoso2017laplacian}. We also show that the bound can be improved for trees in which all deep vertices (Definition~\ref{def:deep}) have high degree.

\begin{definition}\label{def:deep}
   A vertex of a tree is called a \emph{leaf} if its degree is $1$, \emph{penultimate} if it is adjacent to a leaf and \emph{deep} if it is neither. We denote the number of penultimate vertices of a tree $T$ by $p(T)$.
\end{definition} 

\begin{theorem}\label{thm:2}
    If $T$ is a tree in which all deep vertices have a degree at least $k$, $k\geqslant 3$, \[\gamma(T)<\left(1+\frac{1}{(k-2)(k+1)}\right)p(T).\]
\end{theorem}

Since $p(T)\leqslant \mu(T)$ (Lemma~\ref{lem:ubpenultimate}), this result improves the upper bound in Theorem~\ref{thm:1} for every $k \geqslant 3$. The upper bound in Theorem~\ref{thm:1} is established by analysing an extended version (Algorithm~\ref{alg:2}) of an algorithm (Algorithm~\ref{alg:1}) by Braga, Rodrigues, and Trevisan \cite{braga2013distribution}. While Algorithm~\ref{alg:1} computes the number of Laplacian eigenvalues less than, equal to and greater than a certain value, our extension also constructs a dominating set of the tree enroute. We tried to apply the same technique to improve the upper bound of $2$ on the ratio $\nu(T)/\gamma(T)$ for trees \cite{cardoso2017laplacian}. 
However, we discover a family of caterpillars (trees with a dominating path) for which this ratio approaches $2$ (c.\ f.\ Section~\ref{sec:laplacian2}).

\paragraph{Background.} We now list some of the known connections between $m_G I$ for different choices of $I$ and combinatorial parameters of graphs. In 1990, Grone, Merris, and Sunder \cite{grone1990laplacian} proved that $m_G[0,1)\geqslant q(G)$, where $q(G)$ is the number of vertices in $G$ which are not pendant but adjacent to a pendant vertex. In 2013, Braga, Rodrigues, and Trevisan \cite{braga2013distribution} demonstrated that for any tree $T$ with more than one vertex, the inequality $m_T[0,2)\geqslant \lceil\frac{n}{2}\rceil$ holds. They further conjectured that $m_T[0,\overline{d})\geqslant \lceil\frac{n}{2}\rceil$, where $\overline{d}$ represents the average degree of $T$. This conjecture was subsequently confirmed by Jacobs, Oliveira, and Trevisan \cite{jacobs2021most}, and Sin \cite{sin2020number} independently in 2021. In 2022, Ahanjideh et al.\ \cite{ahanjideh2022laplacian} established that, for a connected graph $G$, $m_G[0,1)\leqslant \alpha(G)\leqslant m_G[0,n-\alpha(G)]$, where $\alpha(G)$ is the independence number of $G$. They also characterized graphs for which $m_G[0,1)=\alpha(G)=\frac{n}{2}$. Additionally, they discovered that for connected triangle-free or quadrangle-free graphs $G$, $m_G(n-1,n]\leqslant 1$, and characterized graphs such that $m_G[0,1]=n-1$ or $n-2$. Further, they showed that in connected graphs, $m_G(n-1,n]\leqslant \chi(G)-1$, where $\chi(G)$ is the chromatic number of $G$. Later, Akbari et al.\ \cite{akbari2023classification}, and Cui and Ma \cite{cui2024laplacian} derived several bounds on $m_GI$ based on various graph parameters. Also in 2024, Xu and Zhou explored the Laplacian eigenvalue distribution relative to the diameter of the graph \cite{xu2024dia,xu2024laplacian}. Recently, Zhen, Wong, and Songnian Xu \cite{zhen2025laplacian} and Leyou Xu and Zhou \cite{xu2025girth} studied how the Laplacian eigenvalue distribution relates to the girth of a graph, while Akbari, Alaeiyan, and Darougheh \cite{akbari2025laplacian} studied it in terms of degree sequence. In 2025, Guo, Xue and Liu \cite{guo2025laplacian} proved that for a tree $T$ with diameter $d$, $m_T[0,1)\geqslant\frac{d+1}{3}$, and characterized trees that achieve this bound. Moon and Park \cite{moon2025laplacian} gave a lower bound on $m_G[0,1)$ for unicyclic graphs in terms of its diameter and girth.

\paragraph{Organization of the paper.}
Section \ref{sec:sylvester} examines an algorithm by Braga, Rodrigues, and Trevisan \cite{braga2013distribution} which can be used to count the number of Laplacian eigenvalues of a tree in any given interval. This is an extension of an algorithm by Jacobs and Trevisan \cite{jacobs2011locating}, which does the same for the adjacency spectrum of trees.
Section \ref{sec:proof1} details the proof of Theorem~\ref{thm:1} and Section \ref{sec:example} demonstrates the tightness of the same by constructing an infinite family of trees where the ratio $\gamma(T)/\mu(T)$ approaches $4/3$. 
Section \ref{sec:proof2} proves Theorem~\ref{thm:2}. Reading the proof of Theorem~\ref{thm:2} first may help in following the proof of Theorem~\ref{thm:1} more easily.  Section \ref{sec:laplacian2}, presents an infinite family of trees for which the ratio $\nu(T)/\gamma(T)$ approaches $2$. Section \ref{sec:conclusion} presents some final remarks and discusses a future direction.

\section{Sylvester's Law of Inertia}\label{sec:sylvester}

The \emph{inertia} of a matrix is the number of its positive, negative, and zero eigenvalues, counting multiplicities.
Two matrices $A$ and $B$ over a field are called \emph{congruent} if there exists an invertible matrix $P$ over the same field such that $P^TAP=B$.

\begin{lemma}[Sylvester's Law of Inertia \cite{horn2012matrix}]
    Two real symmetric matrices are congruent if and only if they have the same inertia.
\end{lemma}

Sylvester's Law of Inertia helps one to count the number of eigenvalues of a real symmetric matrix $M$ in any interval $I$. For example, if $I = [a,b)$, the number of eigenvalues of $M$ can be found from the inertia of the matrices $M - aI$ and $M-bI$. In particular $\mu(G)$ is equal to the number of negative eigenvalues of $L(G) - I$. We use Algorithm~\ref{alg:1} by Braga, Rodrigues, and Trevisan \cite{braga2013distribution} which determines the inertia of the matrix $L(T) + \alpha I$ for a tree $T$ by doing a single bottom-up scan of $T$ without using any matrix computations. This algorithm outputs a diagonal matrix that is congruent to $L(T)+\alpha I$. Each entry in the diagonal of the matrix corresponds to the values of the respective vertices at the end of the algorithm.

\begin{algorithm}
    \caption{ \cite{braga2013distribution}}\label{alg:1}
    \textbf{Input:} Rooted Tree $T$, scalar $\alpha$ \\
    \textbf{Output:} Diagonal matrix $D$ congruent to $L(T) + \alpha I$
    
    \begin{algorithmic}[1] 
        \State Initialize $f(v) := d(v) + \alpha$ for all vertices $v$
        \State Order vertices as $v_1,\dots,v_n$ from the lowest level to the highest level (root)
        \For{$k = 1$ to $n$}
            \If{$v_k$ is a leaf}
                \State \textbf{continue}
            \ElsIf{$f(c) \neq 0$ for all children $c$ of $v_k$}
                \State $f(v_k) := f(v_k) - \sum\limits_{\substack{c}} \frac{1}{f(c)}$, where the sum over all children of $v_k$.
            \Else
                \State select one child $v_j$ of $v_k$ for which $f(v_j) = 0$
                \State $f(v_k) := -\frac{1}{2}$, $f(v_j) := 2$
                       \If{$v_k$ has a parent $v_t$}
                             \State remove the edge $v_kv_t$
                        \EndIf
            \EndIf
        \EndFor
    \State Let $D$ be a diagonal matrix with $d_{ii} = f(v_i)$ for each vertex $v_i$ in $T$
    \State \textbf{return} $D$
    \end{algorithmic}
\end{algorithm}

\begin{theorem}[\cite{jacobs2011locating}]
    \label{thm:diagonalize}
   Let $D$ be the diagonal matrix produced by Algorithm~\ref{alg:1} applied to a tree $T$ rooted at any of its vertices and scalar $\alpha=-\beta$. Then, $m_T(\beta,n]$, $m_T[0,\beta)$,$m_T[\beta]$ are equal to the number of positive, negative, and zero values on the diagonal of $D$.
\end{theorem}

\begin{definition}
   The final value $f(v)$ assigned to a vertex $v$ by Algorithm~\ref{alg:1} is referred to as the \emph{inertia} of $v$. The vertex $v$ is called \emph{negative} if its inertia $f(v)$ is negative. We will say that a vertex $u$ is a \emph{zero-child} of its parent $v$ if $f(u) = 0$ when Algorithm~\ref{alg:1} starts to process $v$, irrespective of whether the inertia of $u$ is $0$ or $2$. 
\end{definition}

The main tool in the paper is an analysis of an extension of this algorithm (Algorithm~\ref{alg:2}). We will make repeated use of the following direct corollary of Theorem~\ref{thm:diagonalize} since our extended algorithm simulates a run of Algorithm~\ref{alg:1} with $\alpha = -1$.

\begin{corollary}
    \label{cor:mu}
    For a tree $T$, $\mu(T)$ is equal to the number of negative vertices in $T$ after a run of Algorithm~\ref{alg:1} on $T$ with $\alpha = -1$.
\end{corollary}

\section{Proof of the Main Result}\label{sec:proof1}

In 1990, Grone et al.\ \cite{grone1990laplacian} proved that for any graph $G$, $\mu(G)\geqslant q(G)$ where $q(G)$ is the number of vertices in $G$ which are not pendant but adjacent to a pendant vertex. The following Lemma follows from the same. We also give a direct proof for trees using Algorithm~\ref{alg:1} and Theorem~\ref{thm:diagonalize}.

\begin{lemma}\label{lem:ubpenultimate}
    For any tree $T,\ \mu(T)\geqslant p(T)$.
\end{lemma}

\begin{proof} 
    When $\alpha$ is $-1$,  Algorithm~\ref{alg:1} initializes all the leaf vertices with a value of $0$ and hence every penultimate vertex gets an inertia of $-1/2$.
\end{proof}

\subsection{Handling Adjacent Two-degree Vertices}

In this section, we show that any upper bound (of at least $1$) on $\gamma(T)/\mu(T)$ that one can prove on trees with no adjacent $2$-degree vertices can be extended to all trees. This helps us in simplifying our later proofs.

\begin{definition}
    A path in a graph $G$ is called \emph{clean} if all its internal vertices have degree exactly two in $G$.    
\end{definition}

\begin{definition}
    The \emph{contraction} of a set of vertices $S\subseteq V(G)$ in $G$ replaces the vertices in $S$ with a single new vertex, which is adjacent to the neighbors of all the vertices in $S$.
\end{definition}

\begin{lemma}\label{lem:ratios}
    Let $T'$ be a tree obtained from a tree $T$ by contracting a clean path of length three. Then 
        \[\frac{\gamma(T)}{\mu(T)} \leqslant \frac{\gamma(T')}{\mu(T')}.\]
\end{lemma}

\begin{proof}     
    We shall first show that $\mu(T) = \mu(T') + 1$. To analyze Algorithm~\ref{alg:1} without splitting cases when a vertex gets an $f$-value of zero, we will set $1/0 = \infty$ and work with the extended real number system. This only results in the following changes in Algorithm \ref{alg:1}. If a vertex $v$ has a zero-child, then instead of setting $f(v)=-1/2$, we will get $f(v)=-\infty$.
    The modified algorithm will also change the inertia of one of the children from zero to $2$. However, we need not remove the edge from $v$ to its parent, since the parent gets a contribution of $-1/f(v)=-1/\infty=0$ from $v$. This simulates the effect of deleting the edge from $v$ to its parent. Therefore, this alternate analysis doesn't affect the final sign of $f(v)$ for any vertex $v$. This approach is used only in the proof of this Lemma.
    
    Let $(a, b, c, d)$ be a clean path in $T$ with $b$ being a child of $a$. Let $b_1, \dots, b_l$ be the other children of $a$, and $e_1, \dots, e_k$ be the children of $d$. After contracting $(a, b, c, d)$ into a single vertex $e$, the children of $e$ are $b_1, \dots, b_l, e_1, \dots, e_k$. 
    We will use $f_T$ and $f_{T'}$ respectively to denote the inertia computed on $T$ and $T'$. Since $f_T(b) = 1 - \frac{1}{f_T(c)}$ and $f_T(c) = 1 - \frac{1}{f_T(d)}$, we see that $1 - \frac{1}{f_T(b)} = f_T(d)$. Using this (in the third step) we can see that

     \[\begin{aligned}
        f_T(a) & = (l+1) - \left(\frac{1}{f_T(b)} + \sum_{j=1}^l \frac{1}{f_T(b_j)} \right) \\
            & = \left(1 - \frac{1}{f_T(b)}\right) + \left(l - \sum_{j=1}^l \frac{1}{f_T(b_j)}\right) \\
             & = f_T(d) + \left(l - \sum_{j=1}^l \frac{1}{f_T(b_j)}\right) \\ 
             &= \left( k - \sum_{i=1}^k \frac{1}{f_T(e_i)} \right) + \left( l - \sum_{j=1}^l \frac{1}{f_T(b_j)} \right) \\
             & = (k +l) - \left(\sum_{i=1}^k \frac{1}{f_T(e_i)} + \sum_{j=1}^l \frac{1}{f_T(b_j)} \right) \\
             & = f_{T'}(e).
    \end{aligned}\]
    
    Thus, this contraction does not affect the execution of the algorithm in other parts of the tree. Hence, all vertices of $T'$ other than $e$ retain their inertia as in $T$. One can verify that either $f(d) < 0$ or $f(c) < 0$ (when $0 \leqslant f(d) < 1$) or $f(b) < 0$ (when $f(d) \geqslant 1$). 
    Therefore, $\mu(T)=\mu(T')+1$.

    Next we show that $\gamma(T) \leqslant \gamma(T') + 1$. Let $D'$ be a minimum dominating set of $T'$. If $e \in D'$, then $D = (D' \setminus \{e\}) \cup \{a, d\}$ is a dominating set of $T$. If $e \not\in D'$ then $D = D' \cup \{x\}$, where $x$ is the vertex in $\{b, c\}$ which is at a distance of $3$ from the vertex in $D'$ which dominates $e$ is a dominating set of $T$. 
    Therefore, $\gamma(T) \leqslant \gamma(T') + 1$. 

    Hence
    \[\frac{\gamma(T)}{\mu(T)}\leqslant\frac{\gamma(T')+1}{\mu(T')+1}\leqslant \frac{\gamma(T')}{\mu(T')},\]
    where the second inequality is a simple observation from the fact that
    $\gamma(T')/\mu(T') \geqslant 1$ (Lemma~\ref{lem:ubmu}).  
\end{proof}

\subsection{A New Dominating Set Algorithm}
\label{sec:NewDom}
Let $\mathcal{T}$ denote the set of all finite trees with no adjacent $2$-degree vertices. We design a new algorithm (Algorithm~\ref{alg:2}) which simulates a run of Algorithm~\ref{alg:1} with $\alpha=-1$ on a tree $T \in \mathcal{T}$ and constructs a subset $D$ of $V(T)$ in the same pass using a weight propagation. We will then show that $D$ is a dominating set of size at most $\mu(T)+\frac{1}{3}(p(T)-1)$. Note that, since Algorithm~\ref{alg:2} simulates Algorithm~\ref{alg:1} with $\alpha = -1$, the number of negative vertices in $T$ at the end of Algorithm~\ref{alg:2} is equal to $\mu(T)$. Recall that, by Lemma~\ref{lem:ratios}, it is enough to establish an upper bound for trees in $\mathcal{T}$ in order to obtain an upper bound for all trees.

Throughout this section, we fix an arbitrary tree $T$ from $\mathcal{T}$ and the set computed by Algorithm~\ref{alg:2} on $T$ as $D$. We start by making two observations on the run of Algorithm~\ref{alg:2} on $T$ before proving the key result that $D$ is a dominating set of $T$ (Lemma~\ref{4by3domination}).

\begin{algorithm}[!ht]
    \caption{}\label{alg:2}
    \textbf{Input:} Tree $T\in\mathcal{T}$\\
    \textbf{Output: } A dominating set of $T$.
    
    \begin{algorithmic}[1] 
        \State Root $T$ at a penultimate vertex.
        \State Initialize $f(v) := d(v) -1$ for all vertices $v$, $D :=\emptyset$ and \[ w(v) := 
            \begin{cases}
                4/3,  &  \text{if $v$ is a penultimate vertex}\\
                0, & \text{otherwise}
            \end{cases}
            \] 
        \State Process the non-leaf vertices bottom-up so that each vertex is processed after all its children (all leaves are assumed to be processed).
        \For{a non-leaf vertex $v$}
            \If{$f(c) \neq 0$ for all children $c$ of $v$}
                \State $f(v) := f(v) - \sum\limits_{\substack{c}} \frac{1}{f(c)}$, where the sum over all children of $v$.
            \Else
                \State select one child $c$ of $v$ for which $f(c) = 0$
                \State $f(v) := -\frac{1}{2}$, $f(c) := 2$
                \State if $v$ has a parent $u$, remove the edge $vu$
            \EndIf
            \If{$f(v)<0$ \textbf{and} $v$ is not a penultimate vertex}
                \State $w(v):=w(v)+1$
            \EndIf
            \If{[($w(v)\geqslant 4/3$) or ($f(v)=0$ and $w(v)\geqslant 1$)] \textbf{and} at least one child of $v$ is not dominated by $D$} 
                   \State $D := D\cup \{v\}$ and push a weight of ($w(v)-1$) to $v$'s parent. 
                \Else
                    \State Push a weight of $w(v)$ to $v$'s parent 
                \EndIf
        \EndFor
    \State \textbf{return} $D$
    \end{algorithmic}
\end{algorithm}

\begin{observation}\label{obs:3.1}
    Every non-leaf vertex $v$ pushes a weight of at least $1/3$ to its parent unless $v$ is added to $D$ and $v$ is a zero-child of its parent.
\end{observation}

\begin{proof} 
    We prove this by induction on the distance $d_v$ of $v$ to the nearest leaf below it. The base case is when $v$ is a penultimate vertex ($d_v = 1)$. In this case, $v$ pushes a weight of $4/3 - 1 = 1/3$ to its parent, confirming the claim. 
    Let $v$ be a vertex with $d_v > 1$ and let us assume that for every vertex $x$ with $d_x < d_v$, the claim is true. In particular, the claim is true for every child of $v$. 
    
    If every child of $v$ is a zero-child and all of them are added to $D$, then $f(v) = -1/2$ and $v$ is not added to $D$.
    Hence $w(v) \geqslant 1$ is pushed to its parent. Otherwise, $v$ gets a total weight of at least $1/3$ from its child(ren) which is pushed to its parent unless $v$ is added to $D$.
    The vertex $v$ will be added to $D$ only if either $w(v) \geqslant 4/3$, in which case a weight of at least $1/3$ is pushed to its parent or if $f(v) = 0$ when processing $v$, in which case $v$ is exempted from pushing any weight to its parent $u$ as it is a zero-child of $u$ which is added to $D$.
\end{proof}

\begin{observation}\label{obs:3.2}
    If $v$ is a negative vertex with no zero-child, then it pushes a weight of at least $2/3$ to its parent.
\end{observation}

\begin{proof}
    Since $v$ is a negative vertex, $w(v)\geqslant 1$.     
    If $v$ has at least two children, then according to Observation~\ref{obs:3.1}, it receives a weight of at least $1/3$ from each child since none of them is a zero-child. Consequently, the total weight of $v$ is at least $5/3$, allowing it to push a weight of at least $2/3$ to its parent.
    
    Suppose $v$ has only one child, say $x$, then $x$ is not a leaf since $v$ has no zero-child. Moreover, $x$ must have at least two children since $T$ does not have two adjacent $2$-degree vertices. If $x$ is dominated when $v$ is being processed, then $v$ pushes all its weight to its parent. Otherwise, no child of $x$ is in $D$, and so all of them push at least $1/3$ weight to $x$, which in turn pushes this weight to $v$. As a result, $v$ has a weight of at least $5/3$ and thus pushes at least $2/3$ to its parent.
\end{proof}

\begin{lemma}\label{4by3domination}
    $D$ dominates $T$. 
\end{lemma}

\begin{proof}
    For a vertex $v \in V(G)$, the notation $N[v]$ denotes the set of vertices adjacent to $v$, including $v$ itself. Since all the penultimate vertices are added to $D$, all the penultimates (which includes the root of $T$), and the leaves are dominated. Let $v$ be any non-root deep vertex of $T$ and $u$ be its parent. We will verify that the total weight entering $N[v]$ (including both initialization and propagation) is at least $4/3$, or at least $1$ with $f(u)=0$ when $u$ is being processed. If no vertex in $N[v] \setminus \{u\}$ is added to $D$, then the total weight entering $N[v]$ reaches $u$. Since $u$ has an undominated child ($v$), and it satisfies the weight and $f(u)$ condition, $u$ will be added to $D$. This ensures the inclusion of at least one vertex of $N[v]$ in $D$,  guaranteeing that $v$ is dominated by $D$.

    We call the children of any vertex in $N[v]$ that do not belong to $N[v]$ as \emph{feeders} of $v$. These feeders consist of the siblings and grandchildren of $v$. The parents of these feeders will be called \emph{feeder inputs} of $v$. Every deep vertex in $T$ has either at least two children or a sibling. Hence, $v$ has at least two feeder inputs and two feeders. If any of the feeders is a zero-child, then its parent is negative and gets an initial weight of $1$. The remaining feeder inputs (at least one exists) gets a weight of at least $1/3$ from its feeders (Observation~\ref{obs:3.1}). Hence, $N[v]$ gets a total weight of at least $4/3$ and we are done. So we can assume henceforth that no feeder is a zero-child. Hence, every feeder pushes a weight of at least $1/3$ to its parent (Observation~\ref{obs:3.1}).
    Let $F$ denote the set of feeders of $v$. We split the analysis based on $\lvert F\rvert$. 
    
    \begin{case}\label{case1}
        $\lvert F\rvert\geqslant 4$
    \end{case}
    
    Since each feeder pushes a weight of at least $1/3$ to its parent, $N[v]$ gets a total weight of at least $4/3$.

    \begin{case}\label{case2}
        $\lvert F\rvert=3$
    \end{case}
    
    If $v$ has a feeder $z$ which is negative and has no zero-child, then $z$ pushes a weight $2/3$ to its parent (Observation~\ref{obs:3.2}). Hence $N[v]$ gets a total weight of at least $4/3$ and we are done. So we can assume that every feeder of $v$  either has a positive inertia or has a zero-child. We call a feeder vertex which either has a positive inertia or a zero-child as \emph{pseudo-positive}. 
    Since $\lvert F\rvert =3$, the feeders push a weight of at least $1$ to $N[v]$. It is enough to show that there is either a negative vertex in $N[v]$ or $f(u) = 0$ when $u$ is processed.      
    The key observation is that if all the children of a vertex $x$ are pseudo-positive, then $f(x)\leqslant deg(x)-1$. For every child $x$ of $v$, since all the children of $x$ are pseudo-positive (being feeders), we conclude that $f(x)\in(0,deg(x)-1]$.
    
    There are four subcases to consider - two when $v$ has one sibling and two grandchildren, and two when $v$ has no siblings but has three grandchildren (See Figure~\ref{fig:cases}).
    
    \begin{enumerate}
        \item $v$ has one sibling and two children $x_1$, $x_2$ with one child each.
        
        In this case $f(x_1),f(x_2)\in (0,1]$ and hence $f(v)\leqslant 0$. Thus either $v$ or $u$ is negative.
    
        \item $v$ has one sibling $v_1$ and one child $x_1$ with two children.
        
        In this case $f(x_1)\in (0,2]$ and $f(v)\leqslant \frac{1}{2}$. If $f(v)\leqslant 0$, then we are done. Otherwise, we will use the fact that $v_1$ is pseudo-positive. If $v_1$ has a zero-child, then $f(u)=2-\frac{1}{f(v)}\leqslant 0$.
        
        If $f(v_1)>0$, $f(u)=2-\frac{1}{f(v)}-\frac{1}{f(v_1)}<0$.
    
        \item $v$ has no siblings but three children $x_1$, $x_2$, $x_3$ with one child each.
        
        In this case $f(x_i)\in (0,1], \forall i\in \{1,2,3\}$ and hence $f(v)\leqslant 0$.
        \item $v$ has no siblings but two children $x_1$, $x_2$, with one child for $x_1$ and two for $x_2$.
        
       In this case $f(x_1)\in (0,1]$ and $f(x_2)\in (0,2]$. Hence $f(v)\leqslant \frac{1}{2}$. So either $f(v)\leqslant 0$ or $f(u)<0$.
    \end{enumerate}

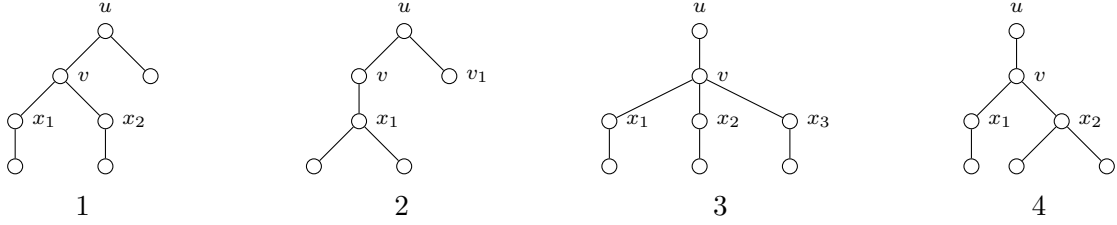
\begin{figure}
    \centering
    \begin{subfigure}[b]{0.23\textwidth}
        \centering
        \begin{tikzpicture}[>=stealth, every node/.style={circle, draw=black, fill=white, inner sep=2pt}, level distance=6mm, sibling distance=12mm]
        
            \node[label=above:{\scriptsize$u$}] {}
                child {node[label=right:{\scriptsize$v$}] {}
                    child {node[label=right:{\scriptsize$x_1$}] {}
                       child {node {}}
                    }
                    child {node[label=right:{\scriptsize$x_2$}] {}
                       child {node {}}
                    }
                }
                child {node {}};
        \end{tikzpicture}
        \caption*{1}
        \label{subfig:1}
    \end{subfigure}
    \hfill
    \begin{subfigure}[b]{0.23\textwidth}
        \centering
        \begin{tikzpicture}[>=stealth, 
        every node/.style={circle, draw=black, fill=white, inner sep=2pt}, level distance=6mm, sibling distance=12mm]
        
            \node[label=above:{\scriptsize$u$}] {}
                child {node[label=right:{\scriptsize$v$}] {}
                    child {node[label=right:{\scriptsize$x_1$}] {}
                       child {node {}}
                       child {node {}}
                    }
                }
                child {node[label=right:{\scriptsize$v_1$}] {}};
        \end{tikzpicture}
        \caption*{2}
        \label{subfig:2}
    \end{subfigure}
    \hfill
    \begin{subfigure}[b]{0.23\textwidth}
        \centering
        \begin{tikzpicture}[>=stealth,
        every node/.style={circle, draw=black, fill=white, inner sep=2pt}, level distance=6mm, sibling distance=12mm]
             \node[label=above:{\scriptsize$u$}] {}
                child {node[label=right:{\scriptsize$v$}] {}
                    child {node[label=right:{\scriptsize$x_1$}] {}
                        child {node {}}
                    }
                    child {node[label=right:{\scriptsize$x_2$}] {}
                        child {node {}}
                    }
                    child {node[label=right:{\scriptsize$x_3$}] {}
                       child {node {}}
                    }
                };
        \end{tikzpicture}
        \caption*{3}
        \label{subfig:3}
    \end{subfigure}
    \hfill
    \begin{subfigure}[b]{0.23\textwidth}
        \centering
        \begin{tikzpicture}[>=stealth,
        every node/.style={circle, draw=black, fill=white, inner sep=2pt}, level distance=6mm, sibling distance=12mm]
            \node[label=above:{\scriptsize$u$}] {}
                child {node[label=right:{\scriptsize$v$}] {}
                    child {node[label=right:{\scriptsize$x_1$}] {}
                        child {node {}}
                    }
                    child {node[label=right:{\scriptsize$x_2$}] {}
                         child {node {}}
                          child {node {}}
                    }
                };
    \end{tikzpicture}
        \caption*{4}
        \label{subfig:4}
    \end{subfigure}
    
    \caption{Four subcases when $|F|=3$ and all feeders of $v$ are pseudo-positive in the proof of Lemma~\ref{4by3domination}}
    \label{fig:cases}
\end{figure}

    \begin{case}\label{case3}
       $\lvert F\rvert=2$
    \end{case}
    
    Since $v$ has at least two grandchildren, $v$ has no siblings in this case. Let $x_1$, $x_2$ be the children of $v$, each with exactly one child.
    If both grandchildren are negative vertices with no zero-child, then each of them pushes a weight of at least $2/3$ to $N[v]$, and we are done. Hence, at least one of them, say the child of $x_1$ is pseudo-positive. In this case, $f(x_1)\leqslant 1$. If $f(x_1)\leqslant 0$ or $f(x_2)\leqslant 0$, we are done and hence we can assume $f(x_1)\in (0,1]$ and $f(x_2)>0$. Then, $f(v)<1$. Hence either $f(v)\leqslant 0$ or $f(u)<0$.
    
    From Cases \ref{case1}, \ref{case2}, and \ref{case3}, it follows that the weight entering $N[v]$ is at least $4/3$, or at least $1$ with $f(u)=0$ when $u$ is processed. Thus, at least one vertex from $N[v]$ is included in $D$, which implies that $D$ dominates $v$.
\end{proof}

\begin{remark}
As an aside, we note that $D$ is infact a minimum dominating set of $T$. This can be seen by comparing Algorithm~\ref{alg:2} with the following greedy algorithm for finding a minimum dominating set in any tree $T$. Begin by setting $D=\emptyset$ and process the vertices of $T$ bottom up. If a vertex $v$ has at least one child that is not dominated by $D$, add $v$ to $D$. A version of this algorithm is analyzed in \cite{stackoverflow} to show that the process results in a minimum dominating set. 
\end{remark}

\begin{lemma}\label{lem:ubgamma}
    \[
        |D| \leqslant \mu(T)+\frac{1}{3}(p(T)-1).
    \]
\end{lemma}

\begin{proof}  
    According to Lemma~\ref{4by3domination}, the set $D$ is a dominating set. It is easy to see that and $|D|$ is at most the total initial weight. The root vertex $v$ being a penultimate vertex, gets an initial weight of $4/3$ and retains a weight of at least $1/3$. Hence,
    \[
     \lvert D \rvert  \leqslant w(T)-\frac{1}{3} =\frac{4}{3}p(T)+(\mu(T)-p(T))-\frac{1}{3} =\mu(T)+\frac{1}{3}(p(T)-1).
    \]
\end{proof} 

\subsection{Completing the Proof of Theorem~\ref{thm:1}}

    Let $T\in\mathcal{T}$. We run Algorithm~\ref{alg:2} on $T$ to obtain a set $D$. By Lemma~\ref{4by3domination}, $D$ is a dominating set and by Lemma~\ref{lem:ubgamma} the size of $D$ is bounded above by $\mu(T)+\frac{1}{3}p(T)$. Hence we have 
    \[ 
        \gamma(T) 
            <           \mu(T)+\frac{1}{3}p(T) 
            \leqslant   \mu(T)+\frac{1}{3}\mu(T)
            =           \frac{4}{3}\mu(T), 
    \]
    where the second inequality follows from Lemma~\ref{lem:ubpenultimate}.
    
    Thus $\gamma(T)/\mu(T)< 4/3$ for any tree in $\mathcal{T}$. From Lemma~\ref{lem:ratios}, we can extend this upper bound to all trees. The lower bound for the ratio $\gamma(T)/\mu(T)$ is immediate from Lemma~\ref{lem:ubmu}. Therefore, for all trees $T$, 
    \[1\leqslant \frac{\gamma(T)}{\mu(T)}<\frac{4}{3}.\]
    This completes the proof of Theorem~\ref{thm:1}.

\section{Tight Example}\label{sec:example}

Hitherto, the tree with the largest known value of $\gamma/\mu$ was a tree on 65 vertices and it had $\gamma/\mu = 25/24$ \cite{cardoso2017laplacian}. In Figure~\ref{fig:5by4}, we show a tree on $11$ vertices with $\gamma=5$ and $\mu=4$. By adding more leaf vertices to the penultimates, we can find an infinite family of trees with $\gamma=5$ and $\mu=4$. Now, we construct an infinite family of trees to show that our upper bound of $4/3$ is tight.

Consider the tree $T_k$ illustrated in Figure~\ref{fig:4by3}. $T_k$, $k \geqslant 2$, has $k$ isomorphic subtrees hanging from the root. We only illustrate the leftmost subtree in the figure. By analyzing the run of Algorithm~\ref{alg:1} on $T_k$, we can see that the negative vertices are the root vertex $r$ and the penultimate vertices. Hence $\mu(T_k)=3k+1$.
Furthermore, a minimum dominating set of $T_k$ must include at least four vertices from each subtree hanging from $r$. Hence $\gamma(T_k) \geqslant 4k$. It is straightforward to construct a dominating set by picking four vertices from each subtree.
Hence $\gamma(T_k) = 4k$ and  
\[\frac{\gamma(T_k)}{\mu(T_k)}=\frac{4k}{3k+1}.\]
As $k$ increases, this ratio approaches $4/3$ showing that the upper bound in Theorem~\ref{thm:1} is optimal.

\begin{figure}[H]
    \centering
    \begin{minipage}{0.45\textwidth}
        \centering
        \begin{tikzpicture}[>=stealth, every node/.style={circle, draw=black, fill=white, inner sep=2pt}, level distance=7mm, 
        level 1/.style={sibling distance=20mm}, 
        level 2/.style={sibling distance=12mm}]
        
            \node [label=right:{\scriptsize $0$}] {}
                child {  node [label=left:{\scriptsize $2$}] {}
                    child { node [label=left:{\scriptsize $-\frac{1}{2}$}] {}
                        child { node [label=left:{\scriptsize $2$}]{} } }
                    child { node [label=left:{\scriptsize $-\frac{1}{2}$}] {}
                        child { node [label=left:{\scriptsize $2$}]{} } }
                }
                child {  node [label=right:{\scriptsize $2$}] {}
                    child { node [label=right:{\scriptsize $-\frac{1}{2}$}] {}
                        child { node [label=right:{\scriptsize $2$}]{} } }
                    child { node [label=right:{\scriptsize $-\frac{1}{2}$}] {}
                        child { node [label=right:{\scriptsize $2$}]{} } }
                };
        \end{tikzpicture}
        \caption{11 vertex tree with $\gamma=5$ and $\mu=4$. Inertia of each vertex is indicated.}
        \label{fig:5by4}
    \end{minipage}
    \hfill
    \begin{minipage}{0.45\textwidth}
        \centering
        \begin{tikzpicture}[>=stealth, 
     solidnode/.style={circle, draw=black, fill=black, inner sep=2pt}, 
     hollownode/.style={circle, draw=black, fill=white, inner sep=2pt},level distance=7mm, sibling distance=12mm]

        \node (1) [hollownode, label=above:{\scriptsize$r$}, label=right:{\scriptsize$-\frac{1}{2}$}] {}
               child { node (2) [hollownode, label=left:{\scriptsize$2$}, label=above:{\scriptsize$u$}] {}
               child { node (5) [hollownode, label=left:{\scriptsize$\frac{1}{2}$}] {}
               child { node (7) [hollownode, label=left:{\scriptsize$2$}] {}
               child { node (9) [solidnode, label=left:{\scriptsize$-\frac{1}{2}$}] {}
               child { node (11) [hollownode, label=left:{\scriptsize$2$}] {} }
            }
               child { node (10) [solidnode, label=right:{\scriptsize$-\frac{1}{2}$}] {}
               child { node (12) [hollownode, label=right:{\scriptsize$2$}] {} }
                }
            }
         }
               child { node (6) [solidnode, label=right:{\scriptsize$-\frac{1}{2}$}] {}
               child { node (8) [hollownode, label=right:{\scriptsize$2$}] {} }
               }
             }
        child [missing]
        child [missing];

        \draw[dashed] (1) -- ++(0,-0.7);
        \draw[dashed] (1) -- ++(0.7,-0.5);

         \end{tikzpicture}
     \caption{Tree with inertia of vertices. The penultimate vertices are indicated by dark circles.}
     \label{fig:4by3}
    \end{minipage}
\end{figure}

\section{Improving the Upper Bound for High-Degree Trees}
\label{sec:proof2}

Let $\mathcal{T}_k$, $k\geqslant 3$, denote the set of all finite trees in which every deep vertex has degree at least $k$. We first establish an upper bound for the domination number in terms of the number of its penultimate vertices, and then see that this improves the upper bound of Theorem~\ref{thm:1} for such trees.

Consider any tree $T\in\mathcal{T}_k$. We run Algorithm~\ref{alg:3} on $T$ to obtain a set $D$ and proceed to analyse $D$ as in  Section~\ref{sec:NewDom}. This algorithm is much simpler than Algorithm~\ref{alg:2} since it does not rely on Algorithm~\ref{alg:1} and only the penultimate vertices are initialised with a non-zero weight.

\begin{algorithm}
    \caption{}\label{alg:3}
    \textbf{Input:} $k\geqslant 3$ and $T\in \mathcal{T}_k$ \\
    \textbf{Output: } A dominating set of $T$.
    \begin{algorithmic}[1]
        \State Root $T$ at a penultimate vertex
        \State Initialize $D=\emptyset$ and $\epsilon= \frac{1}{(k-2)(k+1)}$,
            \[w(v) := 
            \begin{cases}
               1+\epsilon,   & \text{if $v$ is a penultimate vertex}\\
            0, & \text{otherwise}    
            \end{cases}
            \]
        \State Process the vertices bottom up. 
            \For{a vertex $v$}
                \If{$w(v)\geqslant 1+\epsilon$}
                   \State Add $v$ to $D$ and push a weight of ($w(v)-1$) to $v$'s parent
                \Else
                    \State Push a weight of $w(v)$ to $v$'s parent
                \EndIf
            \EndFor
        \State \textbf{return} $D$
    \end{algorithmic}
\end{algorithm}

\begin{observation}\label{obs:5.1}
    Every non-leaf vertex pushes a weight of at least $\epsilon$ to its parent.
\end{observation}

\begin{proof}
    We prove this by induction on the distance $d_v$ of $v$ to the nearest leaf below it. The base case is when $v$ is a penultimate vertex ($d_v = 1)$, in which case, $v$ pushes $\epsilon$ to its parent. 
    Let $v$ be a vertex with $d_v > 1$ and let us assume that for every vertex $u$ with $d_u < d_v$, the claim is true. 
    If $v\in D$, then $w(v)\geqslant 1+\epsilon$, ensuring that a weight of at least $\epsilon$ is pushed to its parent. Otherwise, by induction hypothesis, it receives a weight of at least $\epsilon$ from its children (none of which are leaves), and all of it is passed on to its parent.  
\end{proof}

\begin{lemma}\label{domk}
    $D$ dominates $T$.
\end{lemma}

\begin{proof}
    Let $v$ be an arbitrary vertex in $T$. We will show that $v$ is dominated by $D$. Since all the penultimate vertices of $T$ are added to $D$, we can assume that no child or grandchild of $v$ is a leaf. We denote parent of $v$ as $u$.
    If $v$ or any of its children is in $D$, then we are done.
    If no child of $v$ is in $D$, by Observation~\ref{obs:5.1}, each grandchild of $v$ pushes a minimum weight of $\epsilon$ to its parent. Consequently, each child of $v$ gets and pushes to $v$ a weight of at least $(k-1)\epsilon$. Thus $v$ gets a weight of at least $(k-1)^2\epsilon$. Since $v$ is not in $D$, this weight is passed to its parent $u$. If $u$ is a penultimate vertex, then we are done. Otherwise, all siblings of $v$ push a minimum weight of $\epsilon$ to $u$ and since there are at least $(k-2)$ siblings for $v$, we have $w(u) \geqslant (k-1)^2\epsilon+(k-2)\epsilon \geqslant 1+\epsilon$, ensuring the inclusion of $u$ in $D$. Therefore $v$ is dominated.
\end{proof}

\begin{proof}[\textbf{Proof of Theorem~\ref{thm:2}}]

Let $T \in \mathcal{T}_k$, $k \geqslant 3$. As the root vertex of $T$ is a penultimate vertex, it retains a weight of at least $\epsilon = 1/((k-2)(k+1))$, and from Lemma~\ref{domk}, \[
     \gamma(T)\leqslant \lvert D \rvert  \leqslant w(T)-\epsilon < w(T)=\left(1+\frac{1}{(k-2)(k+1)}\right)p(T).
    \]
\end{proof}

\begin{remark}
    We can run Algorithm~\ref{alg:3} on any tree $T \in \mathcal T$ (i.e., trees without adjacent $2$-degree vertices) with $\epsilon = 1$. Since every vertex $v$ in $T$ not dominated by a penultimate vertex will have at least two non-leaf grandchildren, either $v$ or one of its child will be in $D$ and hence $D$ will dominate $T$. This, together with Lemma~\ref{lem:ratios}, shows that $\gamma(T)/\mu(T) < 2$ for every tree $T$. However we needed a much more involved argument to establish the optimal upper bound of $4/3$ (Theorem~\ref{thm:1}).  
\end{remark}

\begin{corollary}\label{cor:1}
   For any tree $T$ in which all deep vertices have degree at least $k$, $k\geqslant3$, \[1\leqslant \frac{\gamma(T)}{\mu(T)}<1+\frac{1}{(k-2)(k+1)}.\]
\end{corollary}

\begin{proof}
    Follows from Lemma~\ref{lem:ubmu},  Theorem~\ref{thm:2}, and Lemma~\ref{lem:ubpenultimate}.
\end{proof}

\begin{remark}
    Using Lemma~\ref{lem:ratios}, we can extend Corollary~\ref{cor:1} to include trees which can be obtained from a tree $T  \in \mathcal{T}_k$ by replacing one or more edges of $T$ with paths of length $0$ or $1$ modulo $3$.
\end{remark}

\section{Tightness of the bound $\nu(T)/\gamma(T) < 2$ for trees} \label{sec:laplacian2}

Recall that $\nu(G) = m_G[2,n]$. In this section, we show that the bound $\nu(T) < 2\gamma(T)$ for trees $T$ established by Cardoso, Jacobs, and Trevisan \cite{cardoso2017laplacian} is tight.  

Let $T_n$ be the tree constructed from the path $P_{3n} = v_1 v_2 \dots v_{3n}$, $n\geqslant 2$, by adding a leaf vertex adjacent to every vertex $v_{3i-1}$ for $2 \leqslant i \leqslant n$ (See Figure~\ref{fig:caterpillar}). Root $T_n$ at $v_{3n}$ and run Algorithm~\ref{alg:1} with $\alpha=-2$. The number of Laplacian eigenvalues in the interval $[2,n]$, that is $\nu(T)$, is equal to the number of non-negative vertices at the end of this run. One can verify that all the leaves and the vertices $v_{3i}$, $1 \leqslant i \leqslant n$ end with negative values and the remaining vertices end with positive values. Hence $\nu(T_n) = 2n-1$. It is easy to check that $\gamma(T_n) = n$.

\begin{figure}[H]
    \centering
    \begin{tikzpicture}[>=stealth,
        every node/.style={circle, draw=black, fill=white, inner sep=2pt}]
        
        \foreach \i/\val in {1/{$-1$},2/{$1$},3/{$-1$},4/{$1$},5/{$1$},6/{$-1$},7/{$1$},8/{$1$},9/{$-2$}} {
            \node (v\i) at (\i,0) 
                [label=below right:{\scriptsize $v_{\i}$},
                 label=above:{\scriptsize $\val$}] {}; 
        }

        \foreach \i in {1,...,8} {
            \draw (v\i) -- (v\the\numexpr\i+1\relax);
        }

        \node (l5) at (5,-0.8) 
            [label=right:{\scriptsize $-1$}] {};
        \draw (v5) -- (l5);

        \node (l8) at (8,-0.8) 
            [label=right:{\scriptsize $-1$}] {};
        \draw (v8) -- (l8);
    \end{tikzpicture}
    \caption{$T_3$ with inertia of vertices when rooted at $v_9$.}
    \label{fig:caterpillar}
\end{figure}
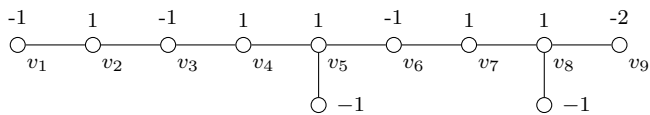

Therefore, \[\frac{\nu(T_n)}{\gamma(T_n)} = 2 - \frac{1}{n}.\] 
As $n$ increases, this ratio approaches 2.
\section{Concluding remarks}\label{sec:conclusion}
In this paper we studied the relationship between the Laplacian spectrum and domination number in trees. The \emph{signless Laplacian matrix} of a graph $G$ is $Q(G) = D(G) + A(G)$ and its spectrum is called the signless Laplacian spectrum of $G$. It is known that \cite{brouwer2011spectra} a graph $G$ is bipartite if and only if the Laplacian spectrum and the signless Laplacian spectrum of $G$ are equal. Therefore our results on trees can also be stated in terms of their signless Laplacian spectrum.

Cardoso, Jacobs and Trevisan \cite{cardoso2017laplacian} also proposed characterizing all graphs $G$ with $\mu(G)=\gamma(G)$. Characterizing trees $T$ with $\mu(T)=\gamma(T)$ is itself interesting and remains open.

\bibliographystyle{unsrt}
\bibliography{ref}

\end{document}